\documentclass{article}
\usepackage{amsmath}
\usepackage{amsthm}
\usepackage{latexsym}
\usepackage{amssymb}
\usepackage{fullpage}

\newcommand{\ii}{{\rm i}}
\newcommand{\bA}{\mathbf{A}}

\newcommand{\MM}{\mathbf{M}}
\newcommand{\bv}{\mathbf{z}}
\newcommand{\bW}{\mathbf{W}}
\newcommand{\bx}{\mathbf{x}}
\newcommand{\barb}{\overline}

\usepackage{verbatim}

\newtheorem{theorem}{Theorem}

\newtheorem{lemma}{Lemma}

\title{Necessary Spectral Conditions for Coloring Hypergraphs} 
\author{Franklin H. J. Kenter}

\begin{document}

\maketitle
\date

\begin{abstract}
Hoffman proved that for a simple graph $G$, the chromatic number  $\chi(G)$ obeys $\chi(G) \le 1 - \frac{\lambda_1}{\lambda_{n}}$ where $\lambda_1$ and $\lambda_n$ are the maximal and minimal eigenvalues of the adjacency matrix of $G$ respectively. Lov\'asz later showed that $\chi(G) \le 1 - \frac{\lambda_1}{\lambda_{n}}$ for any (perhaps negatively) weighted adjacency matrix.

In this paper, we give a probabilistic proof of Lov\'asz's theorem, then extend the technique to derive generalizations of Hoffman's theorem when allowed a certain proportion of edge-conflicts.  Using this result, we show that if a 3-uniform hypergraph is 2-colorable, then $\bar d \le -\frac{3}{2}\lambda_{\min}$ where $\bar d$ is the average degree and $\lambda_{\min}$ is the minimal eigenvalue of the underlying graph. We generalize this further for $k$-uniform hypergraphs,  for the cases $k=4$ and $5$, by considering several variants of the underlying graph.

%
\end{abstract}

\section{Introduction}

Spectral graph theory has produced several results in relation to graph coloring. Let $\lambda_{\max}$ and $\lambda_{\min}$ denote the maximal and minimal eigenvalues, respectively, of the adjacency matrix of $G$. In 1967, Wilf \cite{Wilf} showed that $\chi(G) \le 1+\lambda_{\max},$  and in 1969, Hoffman  showed $\chi(G) \ge 1-\frac{\lambda_{\max}}{\lambda_{\min}}$ \cite{Hoffman}. Hoffman's theorem was improved by Lov\'{a}sz  in 1979 with his work on Shannon capacity by considering any (even negative) weighted adjacency matrix of $G$ \cite{Lovasz}.

We aim to extend these ideas to hypergraphs. However, there are many different notions of eigenvalues for hypergraphs. These notions generally have one of two approaches. The first is to consider some higher-dimensional analog of a matrix. Friedman and Wignerson used a trilinear form to define the spectral gap of a 3-uniform hypergraph \cite{FandW}, and later, Cooper and Dutle extend this concept to hypermatrices using hyperdeterminants \cite{CandD}. The other approach is to dissect the hypergraph into several matrices, and consider these matrices in tandem. Chung used this approach by considering the homology and cohomology chains of hypergraphs. Later, Rodr\'iguez \cite{Rod} followed by Lu and Peng \cite{LandP} considered different walks on the hypergraph with varying degrees of tightness.

The only known result connecting hypergraph coloring to spectra is given by Cooper and Dutle using hypermatrices and hyperdeterminants where they prove an analog of Wilf's theorem for hypergraphs: $\chi(H) \le \lambda_{\max} + 1$  (see \cite{CandD}) . In this paper, we provide necessary spectral conditions for 2-coloring a hypergraph using the second approach to hypergraph spectra as described above. Specifically, our main strategy will be to consider several different graphs based upon the original hypergraph. A more detailed description is given in the next section. Then, we consider these different graphs and their spectra, in tandem, in order to recover information regarding the hypergraph. 

This paper is organized as follows: In section \ref{p} We give definitions and preliminaries. In section \ref{re} we state our main results. Then, we prove the results in two sections. In section \ref{11}, we prove results for graphs, including a probabilistic proof of the Hoffman-Lov\'{a}sz theorem mentioned above, and in section \ref{234}, we adapt these results to various aspects of hypergraphs. Finally, in section \ref{ex}, we give examples of applications of these results.


\section{Preliminaries}\label{p}

A graph $G = (V,E)$ is a set of vertices, $V$,  and edges, $E$, such that every $e\in E$ is an unordered pair of vertices. We say that two vertices, $u, v$, are adjacent, written $u \sim v$ if $\{u,v\} \in E$.  In this paper, we consider undirected graphs which allow for multiple edges between two vertices; however, we do not consider graphs with loops. We define the degree, $d_u$, of a vertex $u$, to be the number of edges containing $u$ where any repeated edge is counted with multiplicity. We let $\bar d$ denote the average degree of $G$. We call a graph $d$-regular if $d_u = d$ for all vertices $u$. 

A hypergraph, $H = (V,E)$ is a set of vertices, $V$,  and edges, $E$, such that every $e\in E$ is a subset of vertices. As with graphs, the degree of a vertex $u$, denoted $d_u$, is the number of hyperedges containing $u$, considering multiplicity, and the average degree will be denoted $\bar d$. We call a hypergraph $r$-uniform if for every edge $e\in E$, $|E| = r$.

We consider traditional vertex coloring of a graph $G$. A proper coloring of a graph with $k$-colors is a function $g: V \to \{1, 2, \ldots k\}$ such that $g(v) \ne g(v')$ whenever $\{v,v'\} \in E$. If such a function exists for a particular integer $k$, we say the graph is $k$-colorable.  For a graph $G$, the chromatic number, $\chi(G)$, is defined to be the least $k$ such that $G$ is $k$-colorable.

In addition, we consider the weak-coloring of a hypergraph. A proper coloring of a hypergraph is a function $h: V \to \{1, 2, \ldots k\}$ such that for every edge $e \in E(H)$, the function $h$  is not constant on $e$. As with graphs, the chromatic number of a hypergraph, $\chi(H)$, is defined to be the least $k$ such that $H$ is $k$-colorable.

A hypergraph, $H$ has an \textit{underlying graph} denoted $G(H)$ which has the same vertex set as $H$, and $e$ is an edge of $G(H)$ if $e \subset f$ for some $f \in E(H)$. For our purposes, we allow for each edge in $G(H)$ to occur once for each hyperedge containing $e$. For example, $G(K_4^3)$, the underlying graph of the complete 3-uniform hypergraph on 4 vertices, is the complete graph on 3 vertices with each edge occurring twice. More generally, the {\it $s$-set graph} of a hypergraph $H$, denoted $G^{(s)}(H)$ has the vertex set $V \choose s$, the set of subsets of $V$ with size $s$, where $\{ \{a_1, \ldots, a_s\} ,\{b_1\ldots b_s\} \}$ is an edge of $G^{(s)}(H)$ if and only if $ \{a_1, \ldots, a_s\}$ and $\{b_1\ldots b_s\} $ are disjoint, and $ \{a_1, \ldots, a_s\} \cup \{b_1\ldots b_s\} \subset f$ for some $f \in E(H).$

For a graph $G$ and hypergraph $H$, we are concerned with relating $\chi(G)$ and $\chi(H)$ with the spectrum of the adjacency matrix (or, as described in the case of $H$, several different adjacency matrices). For a graph $G$, the adjacency matrix, $\mathbf{A}$, is the $|V(G)| \times |V(G)|$ matrix where the entry $\mathbf{A}_{ij}$ is the number of edges between $i$ and $j$. A weighted adjacency matrix of a graph, denoted by $\bW$, is a matrix that satisfies $\bW_{ij} = 0$ whenever $\{i,j\} \not\in E(G)$. For our purposes, we allow for the possibility of negative weights. We denote the largest and smallest eigenvalues of a (real symmmetric) matrix $\MM$, respectively  as $\lambda_{\max}(\MM)$ and $\lambda_{\min}(\MM)$. In the context of hypergraphs, we will let $\lambda_{\min}$ denote the minimum eigenvalue of the underlying graph as described above, and we will let $\lambda^{(s)}_{\min}$ denote the minimum eigenvalue of the adjacency matrix for the $s$-subset graph.

Some of our proofs are probabilistic in nature. The key fact we use is that if a real-valued random variable $X$ obeys $m \le X \le M$ almost surely for some constants $m, M$, then $m \le \mathbb{E}[X] \le M$ where $\mathbb{E}[X]$ denotes the expectation of $X$. In particular, we combine matrix theory and probability in our application of the Courant-Fischer Theorem \cite{MA} in the following way:

If $\mathbf{x} \in \mathbb{C}^n$ is a non-zero random vector, and $\bA$ is a real-symmetric $n \times n$ matrix, then $\lambda_{\min} \le \frac{\mathbb{E}[\mathbf{x}^H \bA \mathbf{x}]}{\mathbf{x}^H \mathbf{x}} \le \lambda_{\max}$.

\section{Main results}\label{re}

We prove the following theorems.

Using a probabilistic proof, we will prove Lov\'asz's theorem:

\begin{theorem}\label{lovasz}
(Hoffman-Lov\'{a}sz 1979) \cite{Lovasz}
For any weighted adjacency matrix (even with negative weights), $\bW$, of a graph $G$

$$\chi(G) \ge 1-\frac{\lambda_{\max}(\bW)}{\lambda_{\min}(\bW)}$$

Specifically, 

$$\chi(G) \ge \max_{\bW} \left( 1-\frac{\lambda_{\max}(\bW)}{\lambda_{\min}(\bW)} \right)$$
where the maximum is taken over all weighted adjacency matrices of $G$
\end{theorem}

Notice that this is one side of the ``sandwich theorem'' in \cite{Lovasz} with regard to the Lov\'{a}sz Theta Function.

We then adapt the method to prove an analogous result if we allow a certain proportion of monochromatic edges:

\begin{lemma}\label{conflicts}
Let $G$ be a graph (perhaps with multiedges) with average degree $\bar d$. If the vertices of $G$ can be (improperly) colored with $k$-colors such that number monochromatic edges is at most $p|E|$, then,

$$k \ge \frac{1-\frac{\lambda_{\min}(\bA)}{\bar d}}{p-\frac{\lambda_{\min}(\bA)}{\bar d}} \text{ or more simply, } p\geq \frac{\bar d+k \lambda_{\min}(\bA)-\lambda_{\min}(\bA)}{\bar d k}$$

Where $\lambda_{\min}(\bA)$ is the smallest eigenvalue of the adjacency matrix of $G$.
\end{lemma}

We remark that the first inequality above, while more complex, explicitly describes the role of $p$ in the result . In particular, when $p = 0$, if $G$ is regular, Hoffman's theorem results. Further, if $p = 1$, then the inequality allows for the graph to be colored with 1 color.

Finally, we apply the previous lemma to produce several results for hypergraphs.

\begin{theorem}\label{3u}
Let $H$ be a 2-colorable 3-uniform hypergraph with average degree $\bar d$. Then,
\[\bar d \le -\frac{3}{2} \lambda_{\min}\]
\end{theorem}

\begin{theorem}\label{4u}
Let $H$ be a 2-colorable 4-uniform hypergraph on $n$ vertices with average degree $\bar d$. Then,
\[\bar d \le -2 \lambda_{\min}- \frac{n-1}{3} \lambda^{(2)}_{\min}\]
\end{theorem}

\begin{theorem}\label{5u}
Let $H$ be a 2-colorable 5-uniform hypergraph on $n$ vertices with average degree $\bar d$. Then,
\[\bar d \le -\frac{5}{2} \lambda_{\min}- \frac{5(n-1)}{12} \lambda^{(2)}_{\min}\]
\end{theorem}

%
%

\section{Proofs of Theorem 1 and Lemma 1}\label{11}

We begin with a proof of Lov\'{a}sz's Theorem in order to demonstrate our technique. We remark that a derandomized variant of this proof can be found in \cite{Niki}. However, we build upon this randomized proof later.

\begin{proof}[Proof of theorem \ref{lovasz}]
For any simple graph $G$, let $\bW$ denote any weighted adjacency matrix for $G$.
Suppose $G$ is $k$-colorable. We will denote  the maximum and minimum eigenvalues of $\bW$ as $\lambda_{\max}$ and $\lambda_{\min}$ respectively. Let $\mathbf{z}$ denote a real unit eigenvector corresponding to $\lambda_{\max}$.
We let $g: V \to \{1,2,\ldots,k\}$ be a $k$-coloring and $\rho: \{1,2,\ldots,k\} \to \mathbb{C}$ be a function assigning each number $0, \ldots, k-1$ a unique $k$-th root of unity. If $\rho$ is chosen randomly and uniformly from all permutations, we may let $\mathbf{x}$ denote a random vector indexed by the colors of the vertices $G$ defined by $\mathbf{x}_j  = \mathbf{z}_j \cdot (\rho \circ g) (j)$.
For any $\mathbf{x}$ we have:
$$\lambda_{\min} \le \frac{\mathbf{x^H W x}}{\mathbf{x^H x}} =\mathbf{x^H W x}$$
So we have,
$$ \lambda_{\min} \le \mathbb{E}[\mathbf{x^H W x}]$$
$$ = \mathbb{E} \left[ \sum_{\{u,v\} \in E} \bW_{uv} (\barb{\bx}_u \bx_v + \barb{\bx_v} \bx_u )\right]$$
By linearity of expectation, we get
$$ = \sum_{\{u,v\} \in E}  \bW_{uv} \mathbb{E} \left[ \barb{\bx_u} \bx_v + \barb{\bx_v} \bx_u \right]$$

where the sum is over all unordered pairs of vertices that form an edge. Thus, we have,

$$ = \sum_{\{u,v\} \in E}  \bW_{uv} \mathbb{E}[(\bv_u \bv_v) ((\rho \circ g) (u) \barb{(\rho \circ g) (v)} +(\rho \circ g) (v) \barb{(\rho \circ g) (u)}] $$
Since $v$ is deterministic, 
$$ = \sum_{\{i,j\} \in E} (\bv_i \bv_j \bW_{uv}) \mathbb{E}[((\rho \circ g) (i) \barb{(\rho \circ g) (j)} +(\rho \circ g) (j) \barb{(\rho \circ g) (i)}]$$
Observe that for any coloring, $\phi$, any permutation of the colors is also a coloring. Hence, the random quantity $((\rho \circ g) (u) \barb{(\rho \circ g) (v)} +(\rho \circ g) (v) \barb{(\rho \circ g) (u)})$ has an equal probability for taking on each of the possible values . Hence, $\mathbb{E}[((\rho \circ g) (u) \barb{(\rho \circ g) (v)} +(\rho \circ g) (v) \barb{(\rho \circ g) (u)} = 2 \sum_{j=1}^{k-1} \exp(2 \pi \ii j /k) = \frac{-2}{k-1}.$ Hence, for the last sum above we get:

\begin{eqnarray*}
&=&  \frac{-1}{k-1} \sum_{\{u,v\} \in E} 2 \bv_u \bv_v \bW_{uv} = \frac{-1}{k-1} \sum_{\{u,v\} \in E} \barb{\bv_u}\  \bW_{uv} \bv_v + \barb{\bv_v}\  \bW_{uv} \bv_u\\
&=& \frac{-1}{k-1} \mathbf{\bv^H \bW \bv} = \frac{-1}{k-1} \lambda_{\max}
\end{eqnarray*}

Altogether we have:
$$\lambda_{\min} \le  \frac{-1}{k-1} \lambda_{\max} $$
And solving for $k$ we get:
$$1-\frac{\lambda_{\max}}{\lambda_{\min}} \le k$$
Choosing $k = \chi(G)$ proves the theorem.
\end{proof}

Many of the previous proofs and the Hoffman-Lov\'{a}sz theorems use discrete methods such as matrix partitioning and covers  \cite{Hoffman} whereas the proof above is analytic which gives us an advantage. We adapt the proof above for lemma \ref{conflicts}:


%
%

\begin{proof}[Proof of lemma \ref{conflicts}]

By hypothesis, suppose there is an improper $k$-coloring, $g: G \to \{1,2,\ldots,k\}$. As in the previous proof, let $\rho:  \{1,2,\ldots,k\} \to \mathbb{C}$ be a function designating each number $1, \ldots, k$ a unique $k$-th root of unity. Similarly, we vary $\rho$ randomly chosen uniformly from all possible permutations, and we let $\mathbf{x}$ denote a random vector determined by the colors of the vertices of $G$ defined by $\mathbf{x}_j  = (\rho \circ g) (j)$.
Let $n$ denote the number of vertices of $G$. 
For any $\mathbf{x}$ we have:
$$\lambda_{\min} (G) \le \frac{\mathbf{x^H \bA x}}{\mathbf{x^H x}} = \frac{\mathbf{x^H \bA x}}{n}$$
So we have,
\begin{eqnarray*}
 n \lambda_{\min} &\le& \mathbb{E}[\mathbf{x^H \bA x}] \\
 &=& \sum_{\{u,v\} \in E}  \bA_{uv} \mathbb{E}[\barb{\bx_u} \bx_v + \barb{\bx_v} \bx_u]
 \end{eqnarray*}
where the sum goes over all unordered pairs of vertices.

Now we consider choosing a random edge, denoted by $\ell$, chosen uniformly from all edges. The monochromatic edges contribute 1 to the sum whereas the monochromatic edges contribute on average $\frac{-1}{k-1}$.  Since the proportion of edges which are monochromatic is at most $p$, 
$$ \sum_{\{u,v\} \in E}  \mathbb{E}[\barb{\bx_u} \bx_v + \barb{\bx_v} \bx_u] \le  \sum_{\{u,v\} \in E}2\frac{p k -1}{k-1}$$

Hence, 

$$n \lambda_{\min} \le  \sum_{\{u,v\} \in E} 2 \bA_{uv} \frac{p k -1}{k-1}=   \frac{p k -1}{k-1} \sum_{\{u,v\} \in E}  2 \bA_{uv}$$
Since $\sum_{\{u,v\} \in E} 2 \bA_{uv} = n \bar d$, we have:
$$\lambda_{\min} \le \frac{\bar d(p k -1)}{k-1}$$
Solving for $k$ (keeping in mind that $\lambda_{\min}$ is negative) yields the result.
\end{proof}

\section{Proofs of the remaining results}\label{234}

\begin{proof}[Proof of Theorem \ref{3u}]
Suppose $H$ is a 2-colorable 3-uniform hypergraph with average degree $\bar d$. If we consider a specific 2-coloring of $H$, observe that for every edge in $H$, there must be 2 vertices of one color and 1 vertex of the other. Hence, the underlying graph of $H$ can be 2-colored if we allow $1/3$ of the edges to be monochromatic. By applying lemma \ref{conflicts}, to the underlying graph which has average degree $2 \bar d$:
$$2 \ge \frac{1-\frac{\lambda_{\min}}{2 \bar d}}{1/3-\frac{\lambda_{\min}}{2 \bar d}}$$
Solving for $\bar d$ yields the result.
\end{proof}

\begin{proof}[Proof of Theorem \ref{4u}]

Suppose $H$ is a 2-colorable 4-uniform hypergraph on $n$ vertices. Let $g: V(H) \to \{0,1\}$ denote a specific 2-coloring of $H$. Observe that for every edge in $H$, there are two cases: 2 vertices of each color; or 3 vertices of one color and 1 of the other. Let $p$ denote the proportion of edges with 3 of one color and 1 of the other. Hence, the underlying graph of $H$ can be 2-colored using $g$ if we allow $1/3+p/6$ of the edges to be monochromatic. 
$$2 \ge \frac{1-\frac{\lambda_{\min}}{3 \bar d}}{1/3+p/6-\frac{\lambda_{\min}}{3 \bar d}}$$
This reduces to $p \le 1+ \frac{\lambda_{\min}}{\bar d}$.

Next, we consider the 2-subset graph of $H$,  $G^{(2)}(H)$. Note that the average degree is $ \frac{3 \bar d}{n-1}$. Let $h: G^{(2)}(H) \to \{0,1\}$ be an (improper) 2-coloring of $G^{(2)}(H)$ as follows: If $\{a,b\} \in V(G^{(2)}(H))$, let $h(\{a,b\})=0$ if $g(a) = g(b)$, and let $h(\{a,b\})=1$ otherwise. For this 2-coloring, the proportion of monochromatic edges is $1-p$. Hence by applying lemma \ref{conflicts},

$$2 \ge \frac{1-\frac{(n-1) \lambda^{(2)}_{\min}}{{24d}}}{1-p-\frac{(n-1) \lambda^{(2)}_{\min}}{{24d}}}$$.

This reduces to \[ \frac{1}{2} - \frac{(n-1) \lambda^{(2)}_{\min}}{6 \bar d} \le p\]

Altogether, we have:

$$ \frac{1}{2} - \frac{(n-1) \lambda^{(2)}_{\min}}{6 \bar d} \le p \le 1+ \frac{\lambda_{\min}}{\bar d}$$

Eliminating the intermediary $p$, and solving for $\bar d$ yields the result. 
\end{proof}

\begin{proof}
[Proof of Theorem \ref{5u}]

The proof follows with exactly the same method as the proof for Theorem \ref{4u}. Let $p$ be the number of hyperedges with 3 of one color and 2 of another. Then, the underlying graph can be 2-colored with $3/5-p/5$ of the edges as monochromatic, and the 2-subset graph can be 2-colored with at most $1/5+2p/5$ monochromatic edges. The average degree of the underlying graph is $4\bar d$, and the average degree of the 2-subset graph is ${20 \bar d} / {(n-1)}$. The remainder follows just as the previous proof and is omitted.
\end{proof}

\section{Examples}\label{ex}


Theorem \ref{3u} is tight. For example, consider the complete 3-uniform hypergraph on 4 vertices, $K_4^3$. However, for the complete 3-uniform graph on 5 vertices $K_5^3$, we have $\bar d = 6$. Also, the corresponding underlying graph is the complete graph where each edge has weight 3, so $\lambda_{\min} = -3$. Hence,  $\frac{3}{2} \lambda_{\min} < \bar d$, and so theorem \ref{3u} indicates that $K_5^3$ is not 2-colorable.

We now show theorem \ref{4u} is applicable. Let $K_n^4$ be the complete 4-uniform graph on $n$ vertices. Clearly, $K_n^4$ is not 2-colorable for $n\ge 7$. Note that the average degree of $K_n^4$ is $\Theta(n^3)$. Observe that the underlying graph of $K_n^4$ is the complete graph on $n$ vertices where each edge is duplicated ${n-2} \choose 2$ times. Hence, $\lambda_{\min} = -\Theta(n^2)$. Likewise, observe that the 2-subset graph of $H$ is a strongly regular graph on ${n} \choose 2$ vertices where each vertex has degree ${n-2} \choose 2$.  Also  two disjoint (and hence, adjacent) vertices have ${n-4} \choose 2$ common neighbors and non-adjacent vertices have ${n-3} \choose 2$ common neighbors. Hence, using the formula for strongly-regular graphs \cite{GandR}, $\lambda^{(2)}_{\min} = 6-2n= -\Theta(n)$. If $H_n$ were 2-colorable for large $n$, then by theorem \ref{4u}, $\bar d = \Theta(n^3) \le \Theta(n^2)$ which poses a contradiction. Therefore, theorem \ref{4u} successfully excludes $K_n^4$ graphs from being 2-colorable for large enough $n$.

For a less trivial example, consider $H =(V,E)$ a 4-uniform hypergraph on 18 vertices with where $V= \{1, \ldots 18\}$, and $E = \{ \{a,b,c,d\} \subset V : a+b+c+d \equiv 0 \mod 3 \}$. A simple calculation in {\it MATLAB} shows that $\lambda_{\min}^{(2)} \approx -39.7119$ and $\lambda_{\min} = -45$. Hence if $H$ is 2-colorable, theorem \ref{4u} requires $\bar d = \frac{680}{3} \le  -2 \lambda_{\min} - \frac{17}{3} \lambda_{\min}^{(2)} \approx 315.034$. Therefore, $H$ is not 2-colorable.

 \end{document}